\documentclass{article}

\usepackage{amssymb,amsmath,amsthm,verbatim}
\usepackage{enumerate}
\usepackage{xspace}
\newcommand{\GA}{{\rm GA}}

\newcommand{\EA}{{\rm EA}}
\newcommand{\BA}{{\rm BA}}
\newcommand{\TA}{{\rm TA}}
\newcommand{\SA}{{\rm SA}}
\newcommand{\PA}{{\rm PA}}
\newcommand{\GL}{{\rm GL}}
\newcommand{\GLIN}{{\rm GLIN}}
\newcommand{\SLIN}{{\rm SLIN}}

\newcommand{\Af}{{\rm Aff}}
\newcommand{\Tr}{{\rm Tr}}
\newcommand{\SL}{{\rm SL}}

\newcommand{\D}{{\rm D}}
\newcommand{\Df}{{\rm Df}}

\DeclareMathOperator{\Spec}{Spec}
\DeclareMathOperator{\id}{id}

\newcommand{\A}{\mathbb{A}}
\newcommand{\IC}{\mathbb{C}}
\newcommand{\IF}{\mathbb{F}}
\newcommand{\IN}{\mathbb{N}}

\newcommand{\Ik}{\mathbb{K}}
\newcommand{\K}{\mathbb{K}}

\DeclareMathOperator{\vd}{vdeg}
\newcommand{\llex}{<_{\rm lex}}

\newtheorem{theorem}{Theorem}

\newtheorem{lemma}[theorem]{Lemma}

\newtheorem{claim}[theorem]{Claim}
\newtheorem{corollary}[theorem]{Corollary}
\newtheorem{conjecture}{Conjecture}
\newtheorem{question}[conjecture]{Question}
\newtheorem{problem}[conjecture]{Problem}

\theoremstyle{definition}
\newtheorem{definition}{Definition}

\theoremstyle{remark}
\newtheorem{remark}{Remark}
\newtheorem{example}{Example}

\newcommand{\cotame}{co-tame\xspace}

\title{Normal subgroups generated by a single polynomial automorphism}
\author{%
Drew Lewis\thanks{Department of Mathematics and Statistics,  University of South Alabama.  Email address: \texttt{drewlewis@southalabama.edu}}%
}

\begin{document}
\maketitle

\begin{abstract}
We study criteria for deciding when the normal subgroup generated by a single polynomial automorphism of $\A^n$ is as large as possible, namely equal to the normal closure of the special linear group in the special automorphism group.  In particular, we investigate {\em $m$-triangular automorphisms}, i.e. those that can be expressed as a product of affine automorphisms and  $m$ triangular automorphisms.  Over a field of characteristic zero, we show that every nontrivial $4$-triangular special automorphism generates the entire normal closure of the special linear group in the special tame subgroup, for any dimension $n \geq 2$.  This generalizes a result of Furter and Lamy in dimension 2.

\end{abstract} 

\section{Introduction}
Let $\Ik$ be a field.  One of the fundamental problems in affine algebraic geometry is to try to describe the structure of $\GA_n(\Ik)$, the group of polynomial automorphisms of  $\A^n$.  There are a few natural subgroups:
\begin{itemize}
\item The general linear group $\GL_n(\Ik)$;
\item The affine group $\Af_n(\Ik)$ consisting of automorphisms of degree one;
\item The triangular subgroup $\BA_n(\Ik)$;
\item The subgroup $\EA_n(\Ik)$ generated by elementary automorphisms, i.e. those with unital Jacobian determinant fixing $n-1$ variables;
\item The tame subgroup $\TA_n(\Ik)$ generated by the triangular and affine automorphisms;
\item The special automorphism group $\SA_n(\Ik)$, consisting of automorphisms with unital Jacobian determinant.
\end{itemize}

It is a classical result of Jung and van der Kulk \cite{Jung,vanderKulk} that in dimension two, the tame subgroup is the entire automorphism group, while Shestakov and Umirbaev \cite{Shestakov-Umirbaev} famously showed that this does not hold in dimension three (in characteristic zero); this question, known as the tame generators problem, remains open in higher dimensions.  

A natural area of inquiry is to describe subgroups lying between the affine and the tame subgroup.  In dimension two, there are many such subgroups due to the classical result that $\TA_2(\Ik)$ is an amalgamated free product of $\Af_2(\Ik)$ and $\BA_2(\Ik)$ over their intersection, but in higher dimensions (and characteristic zero; see \cite{Edo-Kuroda} for the positive characteristic case) this is a surprisingly delicate question.  It was not until recently that Edo and the author \cite{Edo-Lewis} gave the first example of such an intermediate subgroup in characteristic zero.  The idea there was to study {\em co-tame automorphisms}, defined by Edo \cite{Edo} as those that together with the affine group generate the entire tame subgroup; the example of \cite{Edo-Lewis} is an automorphism that is tame but not co-tame, which therefore generates a proper intermediate subgroup between $\Af_n(\Ik)$ and $\TA_n(\Ik)$.  Interestingly, Edo \cite{Edo} showed that certain wild maps, including the Nagata map, are co-tame.

One key difficulty in describing this subgroup lattice between the affine and tame subgroups arises from the fact that many simply constructed automorphisms are co-tame.  To describe this difficulty further, let us make a precise definition.

\begin{definition}
A tame automorphism $\phi$ is called {\em $m$-triangular} if it can be written in the form $\phi = \alpha _0 \tau _1 \alpha _1 \cdots \tau _m \alpha _{m}$ for some $\tau _i \in \BA_n(\Ik)$ and $\alpha _i \in \Af_n(\Ik)$.
\end{definition}

The author and Edo \cite{Edo-Lewis17} recently showed that, for $n \geq 3$, all 3-triangular automorphisms are co-tame, while in the $n=3$ case, for all $m \geq 4$ there exist $m$-triangular automorphisms that are not co-tame (and thus generate proper intermediate subgroups between the affine and tame subgroups). 

This phenomenon of single automorphisms generating large subgroups also appears in the work of Furter and Lamy \cite{Furter-Lamy}, who were studying normal subgroups in dimension two with an eye towards establishing the non-simplicity of the two-dimensional Cremona group (later proved over an algebraically closed field by Cantat and Lamy \cite{CantatLamy}).  To be more precise, let us quickly fix some notations.

\begin{itemize}
\item If $H \subset \SA_n(\Ik)$, we use $\langle H \rangle^S$ to denote the normal subgroup generated by $H$ in $\SA_n(\Ik)$.
\item If $H \subset \GA_n(\Ik)$, we use $\langle H \rangle^G$ to denote the normal subgroup generated by $H$ in $\GA_n(\Ik)$.
\item The group $\SLIN_n(\Ik):=\langle \SL_n (\Ik) \rangle ^S$ is the smallest normal subgroup of $\SA_n(\Ik)$ that contains $\SL_n(\Ik)$.
\item The group $\GLIN_n(\Ik):= \langle \GL_n(\Ik) \rangle ^G$ is the smallest normal subgroup of $\GA_n(\Ik)$ that contains $\GL_n(\Ik)$.
\end{itemize}

Danilov \cite{Danilov} showed that $\SA_2(\Ik)$ (for a field of characteristic zero) is not simple by constructing a $13$-triangular map that generates a proper normal subgroup.   Furter and Lamy \cite{Furter-Lamy} showed that the normal subgroup generated by any single nontrivial $4$-triangular automorphism in $\SA_2(\Ik)$ is the entire group $\SA_2(\Ik)$.  Moreover, by taking advantage of the amalgamated free product structure of $\GA_2(\Ik)$, they showed that for $m \geq 7$, generic $m$-triangular automorphisms generate proper normal subgroups.  More recently, the non-simplicity of $\SA_2(\Ik)$ was shown for all fields by Minasyan and Osin \cite{MinasyanOsin}.

In dimension 3 (and characteristic zero), while $\TA_3(\Ik)$ is a proper subgroup of $\GA_3(\Ik)$ \cite{Shestakov-Umirbaev}, the tame subgroup is still an amalgamated free product \cite{Wright} (of three subgroups along their pairwise intersections).  Recently Lamy and Przytycki \cite{LamyPrzytycki} took advantage of this to give a class of examples of $m$-triangular automorphisms $$\phi _m = (x_2,x_1+x_2x_3,x_3)^m(x_3,x_1,x_2)$$ such that $\langle \phi _m \rangle ^{\SA_3(\Ik) \cap \TA_3(\Ik)}$ is a proper subgroup of $\SA_3(\Ik) \cap \TA_3(\Ik)$ for every even $m \geq 12$; moreover, they showed that $\TA_3(\Ik)$ is acylindrically hyperbolic.   However, it remains to our knowledge an open question whether $\langle \phi _m \rangle ^S =\SLIN_3(\Ik)$.

The group $\GLIN_n(\Ik)$ was introduced by Maubach and Poloni \cite{Maubach-Poloni}, who were investigating a weaker form of Meister's Linearization problem\footnotemark:
\begin{problem}\label{prob:Meister}
For which $\phi \in \GA_n(\IC)$ do there exist some $s \in \IC^*$ such that $(sx_1,\ldots,sx_n)\phi$ is conjugate to an element of $\GL_n(\IC)$?
\end{problem}

\footnotetext{We feel obliged to point the reader to Section 8.3 of \cite{ArnoBook}, in which van den Essen gives a delightful accounting of the story of the construction of counterexamples to Meister's original Linearization Conjecture and the related Markus-Yamabe Conjecture.}

While van den Essen \cite{vandenEssen} gave an example of an automorphism that does not have this property (see Example \ref{ex:vdE}), Maubach and Poloni showed that the (wild) Nagata map does have this property, and thus lies in $\GLIN_n(\IC)$.  This led them to make the following conjecture.

\begin{conjecture}\label{con:g}
If $\Ik \neq \IF_2$, then $\GLIN_n(\Ik) = \GA_n(\Ik)$.
\end{conjecture}

This is trivial for $n=1$, and a consequence of the Jung-van der Kulk theorem for $n=2$ (see Theorem \ref{thm:GLIN}), but remains  open for $n \geq 3$.  We remark that Maubach and Willems  showed the necessity of the $\Ik\neq \IF_2$ hypothesis in \cite{Maubach-Willems}. Here, we add the following, slightly stronger conjecture:
\begin{conjecture}\label{con:s}
If $\Ik \neq \IF_2$, then $\SLIN_n(\Ik) = \SA_n(\Ik)$.
\end{conjecture}

In section \ref{secSLIN}, we study the group $\SLIN_n(\Ik)$ in any characteristic, and show that $\SLIN_n(\Ik) = \langle \EA_n(\Ik)\rangle ^S$ for all fields other than $\IF_p$ for a prime $p$.  Since $\TA_n(\Ik) \cap \SA_n(\Ik)=\EA_n(\Ik)$, this motivates us to make the following definition.

\begin{definition}
A special automorphism $\theta \in \SA_n(\Ik)$ is called {\em normally co-tame} if $\langle \theta \rangle ^S \geq  \SLIN_n(\Ik)$.  
\end{definition}

Note that if Conjecture \ref{con:s} is true, then an automorphism $\theta \in \SA_n(\Ik)$ is normally co-tame if and only if $\langle \theta \rangle ^S = \SA_n(\Ik)$.  Thus, in this paper we turn our attention to describing classes of maps that are normally co-tame.  In particular, we generalize a result of Furter and Lamy to all dimensions, and show

\begin{theorem}[Main Theorem]
Over an field of characteristic zero, every nontrivial $4$-triangular automorphism is normally co-tame.
\end{theorem}

We also quickly show that a class of exponential maps, including the (wild) Nagata map, are all normally co-tame (cf. \cite{Maubach-Poloni}).  Finally, we show that a related class consisting of triangular maps composed with exponential maps are all normally co-tame.  In particular, this shows that the example of van den Essen \cite{vandenEssen} that does not satisfy Problem 1 does in fact lie in $\SLIN_n(\Ik)$, lending some more support to Conjectures \ref{con:g} and \ref{con:s}.  

%In contrast to \cite{Furter-Lamy} and \cite{LamyPrzytycki}, we work in any dimension, which means we know less about the structure of $\TA_n(\Ik)$ (as \cite{Shestakov-Umirbaev}, \cite{Umirbaev} and thus \cite{Wright} require characteristic zero).  

\section{Preliminaries}
We begin by recalling some standard definitions; see \cite{vandenEssen} for a general reference on polynomial automorphisms.  We use $\Ik^{[n]}=\Ik[x_1,\ldots,x_n]$ to denote the $n$-variable polynomial ring.

\begin{itemize}
\item $\GA_n(\Ik)$ is the group of automorphisms of $\Spec \Ik^{[n]}$ over $\Spec \Ik$.  It is anti-isomorphic to the group of $\Ik$-automorphisms of $\Ik^{[n]}$.  We abuse this correspondence freely, and for $\phi \in \GA_n(\Ik)$ and $P \in \Ik^{[n]}$ will write $(P)\phi$ for the image of $P$ under the corresponding automorphism of $\Ik^{[n]}$.  By writing the automorphism on the right, the usual composition holds, namely if $\psi \in \GA_n(\Ik)$ as well, then $(P)\phi \psi = ((P)\phi)\psi$.
\item $\Tr_n(\Ik)$ denotes the group of translations.  
\item $\EA_n(\Ik)$ denotes the subgroup generated by elementary automorphisms, i.e. those of the form $$(x_1,\ldots,x_{i-1},x_i+P(x_1,\ldots,x_{i-1},x_{i+1},\ldots,x_n),x_{i+1},\ldots,x_n)$$
for some $P \in \Ik[x_1,\ldots,x_{i-1},x_{i+1},\ldots,x_n]$.
\item $\BA_n(\Ik)$ denotes the subgroup of (lower) triangular automorphisms, i.e. those of the form $$\left(a_1x_1+P_1,a_2x_2+P_2(x_1),\ldots,a_nx_n+P_n(x_1,\ldots,x_{n-1})\right)$$
for some $a_i \in \Ik^*$ and $P_i \in \Ik[x_1,\ldots,x_{i-1}].$
%We say an automorphism is {\em $m$-triangular} if it can be written in the form $\alpha _0 \tau _1 \alpha _1 \cdots \tau _m \alpha_m$ for some $\alpha _i \in \Af_n(\Ik)$ and $\tau _i \in \BA_n(\Ik)$.
\item The tame subgroup is $\TA_n(\Ik) = \langle \EA_n(\Ik), \GL_n(\Ik) \rangle = \langle \BA_n(\Ik), \Af_n(\Ik) \rangle$.  
\item We use $\D_n(\Ik)$ to denote the diagonal subgroup of $\GL_n(\Ik)$, and define $\Df_n(\Ik) = \D_n(\Ik) \ltimes \Tr_n(\Ik)$.  This group consists of all automorphisms of the form 
$$(a_1x_1+b_1,\ldots,a_nx_n+b_n)$$ for some $a_i \in \Ik^*$ and $b_i \in \Ik$.
\item $\PA_n(\Ik)$ is the group of parabolic automorphisms, i.e. those of the form 
$$\left(H_1,\ldots,H_{n-1}, a_nx_n+P_n(x_1,\ldots,x_{n-1})\right)$$
for some $H_i \in \Ik ^{[n-1]}$, $a_n \in \Ik^*$, and $P_n \in \Ik[x_1,\ldots,x_{n-1}]$.  
%This group can be expressed as a semidirect product (\cite{Edo-Lewis17}, Lemma 5).
\end{itemize}

\begin{definition}
We define the {\em vector degree} $\vd: \BA_n  (\Ik) \rightarrow \IN^n$ by writing $\tau = \left(a_1x_1+P_1,a_2x_2+P_2(x_1),\ldots,a_nx_n+P_n(x_1,\ldots,x_{n-1})\right)$ for some $a_i \in \Ik$, $P_i \in \Ik[x_1,\ldots,x_{i-1}]$
and setting 

$$\vd(\tau) = \left( \deg \left( P_1 \right), \ldots, \deg \left(P_n\right) \right).$$
We will, somewhat unusually, adopt the convention that $\deg(0)=0$ for convenience.
\end{definition}

\begin{example}
$\vd\left( (x_1+2,x_2+x_1^2,x_3-x_1^2+x_1x_2^4) \right) = (0,2,5)$.
\end{example}

It will be convenient to order $\IN^n$ lexicographically; we denote this partial order by $\llex$ and write, for example, $(0,2,5) \llex (0,3,3)$.
The utility of the vector degree is made clear by the following lemma.

\begin{lemma}\label{lem:vdreduce}
Let $\tau \in \BA_n (\Ik)$.
\begin{enumerate}
\item We have $\tau \in \Df_n(\Ik)$ if and only if $\vd(\tau)=(0,\ldots,0)$.
\item If $\gamma \in \Tr_n(\Ik)$ and $\tau \notin \Df_n(\Ik)$, then $\vd \left(\tau ^{-1} \gamma \tau\right) \llex \vd \left(\tau\right)$.
\end{enumerate}
\end{lemma}
\begin{proof}
The first statement is immediate from our definition of $\Df_n(\Ik)$.  For the second, write $\tau = (a_1x_1+P_1,a_2x_2+P_2(x_1),\ldots,a_nx_n+P_n(x_1,\ldots,x_{n-1}))$ for some $a_i \in \Ik^*$, $P_i \in \Ik[x_1,\ldots,x_{i-1}]$, and write $\gamma =(x_1+b_1,\ldots,x_n+b_n)$ for some $b_i \in \Ik$.  Since $\tau \notin \Df_n(\Ik)$, we have $(0,\ldots,0) \llex \vd(\tau)$.  Therefore we let $r>1$ be minimal with $\deg P_r >0$, so that $P_1,\ldots,P_{r-1} \in \Ik$.  Then it is easy to see that for $i<r$, $(x_i)\tau ^{-1} \gamma \tau = x_i+\frac{a_i}{b_i}$, and
$$(x_r)\tau ^{-1} \gamma \tau = x_r+\frac{b_r}{a_r}+\frac{1}{a_r}\left( P_r(x_1,\ldots,x_{r-1})-P_r\left(x_1+\frac{b_1}{a_1},\ldots,x_{r-1}+\frac{b_{r-1}}{a_{r-1}}\right)\right).$$
Taylor's theorem then implies $\vd \left(\tau ^{-1} \epsilon \tau\right) \llex \vd \left(\tau\right)$.
\end{proof}

\section{The group $\SLIN_n(\Ik)$}\label{secSLIN}
In this section, our goal is a characterization of the group $\SLIN_n(\Ik)$ in Theorem \ref{thm:SLIN}.  We also prove some useful lemmas along the way.  
We make the following definitions for convenience:  
\begin{definition} Let $1 \leq i,j \leq n$ with $i \neq j$, let $f \in \Ik[x_1,\ldots,x_{i-1},x_{i+1},\ldots,x_n]$, and let $c \in \Ik^*$.  Then we define $\epsilon _{i,f} \in \EA_n(\Ik)$,  $\delta _{i,c} \in \GL_n(\Ik)$, and  $\delta _{i,j,c} \in \SL_n(\Ik)$ by 
\begin{align*}
\epsilon _{i,f} & =(x_1,\ldots,x_{i-1},x_i+f,x_{i+1},\ldots,x_n), \\
\delta _{i,c} &= (x_1,\ldots,x_{i-1},cx_i,x_{i+1},\ldots,x_n), \\
\delta _{i,j,c} &= \delta _{i,c} \delta _{j,c^{-1}} .
\end{align*}
\end{definition}

A direct computation yields the following useful commutator formula.
\begin{lemma}\label{lem:commutator}
Let $1 \leq i , j \leq n$ with $i \neq j$, let $a \in \Ik$ and let $b \in \Ik^*$. Then 
$$\epsilon _{i,a}^{-1} \delta _{i,j,b} \epsilon _{i,a} \delta _{i,j,b}^{-1}  =\epsilon _{i,ab-a}.$$
\end{lemma}

\begin{lemma}\label{lem:tr}
Let $1 \leq i \leq n$, and let $c \in \Ik^*$.  Then $\langle \epsilon_{i,c} \rangle ^S = \langle \Tr_n(\Ik)\rangle ^S $.
\end{lemma}
\begin{proof}
Since $\epsilon _{i,c} \in \Tr_n(\Ik)$, we have $\langle \epsilon_{i,c} \rangle ^S  \leq \langle \Tr_n(\Ik)\rangle ^S$. To show the opposite containment, it suffices to show $\epsilon _{j,d} \in \langle \epsilon _{i,c} \rangle ^S$  for any $d \in \Ik$ and $1 \leq j \leq n$.  

We first claim that $\epsilon _{i,d} \in \langle \epsilon _{i,c} \rangle ^S$.  This is immediate if $d=-c$, as $\epsilon _{-c}=\epsilon _c ^{-1}$.  If $d \neq -c$, choose any $1\leq k \leq n$ with $k \neq i$.  Then by Lemma \ref{lem:commutator}, we have 
$$\epsilon _{i,d}=\epsilon _{i,c}^{-1} \left(\delta _{i,k,1+\frac{d}{c}}, \epsilon _{i,c} \delta _{i,k, 1+\frac{d}{c}}^{-1} \right)\in \langle \epsilon _{i,c}\rangle ^S.$$
We can now assume $j \neq i$ and note that 
$$\epsilon _{j,d} = \epsilon _{i,d}^{-1} \left(\epsilon _{j,x_i} \epsilon _{i,d} \epsilon _{j,x_i}^{-1} \right)\in \langle \epsilon _{i,d} \rangle ^S \leq \langle \epsilon _{i,c} \rangle ^S.$$
\end{proof}

\begin{corollary}\label{cor:tr} If $\gamma \in \Tr_n(\Ik)$ is not the identity, then $\langle \gamma \rangle ^S = \langle \Tr_n (\Ik)\rangle^S$
\end{corollary}
\begin{proof}
Since $\gamma \in \Tr_n(\Ik)$, we immediately have  $\langle \gamma \rangle ^S \leq \langle \Tr_n (\Ik)\rangle^S$.   By Lemma \ref{lem:tr}, in order to show $\langle \Tr_n(\Ik) \rangle ^S \leq \langle \gamma \rangle ^S$, it suffices to show that $\epsilon _{i,c} \in \langle \gamma \rangle ^S$ for some $1\leq i \leq n$ and $c \in \Ik^*$.  Write $\gamma = (x_1+c_1,\ldots,x_n+c_n)$ for some $c_1,\ldots,c_n \in \K$.   Since $\gamma \neq \id$, there exists $1 \leq j \leq n$ with $c_j \neq 0$.  Let  $1\leq i \leq n$ with $i \neq j$, and compute  
$$\epsilon _{i,c_j}=\gamma ^{-1} \left(\epsilon _{i,x_j}  \gamma \epsilon _{i,x_j} ^{-1}\right) \in \langle \gamma \rangle ^S.$$
Thus  $\langle \epsilon _{i,c_j} \rangle ^S \leq \langle \gamma \rangle ^S$ as required.
\end{proof}

\begin{theorem}\label{thm:Tr=SLIN} If $\Ik \neq \IF_2$, then
$\langle \Tr_n(\Ik)\rangle ^S=\SLIN_n(\Ik)$.
\end{theorem}
\begin{proof}
Since $\SL_n(\Ik)$ is generated by elementary matrices, it suffices to show that $\langle \epsilon _{i,ax_j} \rangle ^S = \langle \Tr_n(\Ik) \rangle ^S$ for any $a \in \Ik^*$, $1 \leq i,j \leq n$ with $i\neq j$.  To see that $\langle \epsilon _{i,ax_j}\rangle ^S \geq \langle \Tr_n(\Ik) \rangle ^S$, we observe 
$$\left(\epsilon _{j,1}^{-1} \epsilon _{i,ax_j}\epsilon _{j,1} \right) \epsilon _{i,ax_j}^{-1} = \epsilon _{i,a}$$
and apply Lemma \ref{lem:tr}.  

The opposite containment is somewhat more delicate.  First, suppose $\Ik$ does not have characteristic two.  Then we compute 
$$\epsilon _{i,ax_j} = \epsilon _{i,-\frac{a^2}{4}}\epsilon _{j,-\frac{a}{2}} \left(\epsilon _{i,x_j^2} \epsilon _{j,\frac{a}{2}} \epsilon _{i,x_j ^2} ^{-1} \right)\in \langle \Tr_n(\Ik)\rangle ^S.$$

Now, assume $\Ik$ has characteristic two.  Since $\Ik \neq \IF_2$ by assumption, choose any $b \in \Ik^*$ with $b \neq a$, and set $c=\frac{b^2}{a-b} \in \Ik^*$.  
\begin{claim}
$$\epsilon _{i,ax_j}= \delta _{i,j,c}^{-1} \left(  \epsilon _{j,-cb^3} \left(\epsilon _{i,cx_j^3} \epsilon _{j,b} \epsilon _{i,cx_j ^3} ^{-1}  \right)  \epsilon _{j,-bc^3} \left(\epsilon _{i,bx_j^3} \epsilon _{j,c} \epsilon _{i,bx_j ^3} ^{-1}  \right)\right) \delta _{i,j,c}.$$
\end{claim}
We first note that the claim implies $\epsilon _{i,ax_j} \in \langle \Tr_n(\Ik)\rangle ^S$, completing the proof.  The claim is established by direct computation: first observe that setting $f_1 = cbx_j^2+cb^2x_j$ and $f_2=bcx_j^2+bc^2 x_j$, we have
\begin{align*}
 \epsilon _{j,-cb^3} \left(\epsilon _{i,cx_j^3} \epsilon _{j,b} \epsilon _{i,cx_j ^3} ^{-1}\right) &= \epsilon _{i,f_1} \\
\epsilon _{j,-bc^3} \left(\epsilon _{i,bx_j^3} \epsilon _{j,c} \epsilon _{i,bx_j ^3} ^{-1}\right)&= \epsilon _{i,f_2}
\end{align*}
Therefore, setting $f_3=f_1+f_2=(cb^2+bc^2)x_j$ (here we are using the characteristic two assumption), we have
$$\left( \epsilon _{j,-cb^3} \epsilon _{i,cx_j^3} \epsilon _{j,b} \epsilon _{i,cx_j ^3} ^{-1}  \right) \left( \epsilon _{j,-bc^3} \epsilon _{i,bx_j^3} \epsilon _{j,c} \epsilon _{i,bx_j ^3} ^{-1}  \right) = \epsilon _{i,f_3}.$$
Finally, we observe $$ \delta _{i,j,c}^{-1} \epsilon _{i,f_3} \delta _{i,j,c} = \epsilon _{i,\left(\frac{b^2}{c}+b\right)x_j}=\epsilon _{i,ax_j}.$$
\end{proof}

\begin{remark}
To our knowledge, it remains an open question whether Theorem \ref{thm:Tr=SLIN} holds over $\IF_2$.
\end{remark}

\begin{corollary}\label{cor:singleAffine}
Let $\Ik$ be a field other than $\IF_2$, and let $\alpha \in \Af_n(\Ik) \cap \SA_n(\Ik)$.  If $\alpha \neq \id$, then $\langle \alpha \rangle ^S = \SLIN_n(\Ik)$.
\end{corollary}
\begin{proof}
First, write $\alpha = \lambda \gamma$ for some $\lambda \in \SL_n(\Ik)$ and $\gamma \in \Tr_n(\Ik)$.  
Note that Theorem \ref{thm:Tr=SLIN} implies that $\gamma \in \SLIN_n(\Ik)$, so we thus have $\langle \alpha \rangle ^S \leq \SLIN_n(\Ik)$.  So we are left to show the opposite containment.

If $\lambda = \id$, then Corollary \ref{cor:tr} and Theorem \ref{thm:Tr=SLIN} show $\langle \alpha \rangle ^S = \SLIN_n(\Ik)$; we thus assume $\lambda \neq \id$.
Write $(x_i)\lambda = a_{i,1}x_1 +\cdots+ a_{i,n}x_n$ for some $a_{i,j} \in \Ik$.  Since $\lambda \neq \id$, there is some $1\leq i,j \leq n$ with $a _{i,j} \neq \delta _{i,j}$.  Then setting $\gamma_0 = \left(\epsilon _{j,1} ^{-1} \alpha \epsilon _{j,1}\right) \alpha ^{-1} \in \langle \alpha \rangle ^S$, one easily computes that
$$\gamma_0 = (x_1+a_{1,j}-\delta _{1,j},  \ldots, x_n+a_{n,j}-\delta _{n,j}) \in \Tr_n(\Ik).$$
Note that $\gamma_0 \neq \id$, and $\langle \alpha \rangle ^S \geq \langle \gamma _0 \rangle ^S = \SLIN_n(\Ik)$ (with the last equality following from Corollary \ref{cor:tr} and Theorem \ref{thm:Tr=SLIN}).
\end{proof}

\begin{theorem}\label{thm:SLIN}
Let $\Ik$ be any field other than $\IF_p$ for a prime $p$.  
Then
$\SLIN_n(\Ik)=\langle \EA_n(\Ik)\rangle^S.$
\end{theorem}
\begin{proof}
Since $\SL_n(\Ik)$ is generated by elementary matrices, we have $\langle \EA_n(\Ik)\rangle ^S~\geq~\SLIN_n(\Ik)$.  For the other containment, it suffices to show that $\epsilon _{k,aM} \in \SLIN_n(\Ik)$ for any monomial $M \in \Ik[x_1,\ldots, \hat{x}_k, \ldots, x_n]$, $a \in \Ik^*$, and $1 \leq k \leq n$ (as $\EA_n(\Ik)$ is generated by elementary automorphisms of this form).  Moreover, conjugating by $(-x_k,x_2,\ldots,x_{k-1},x_1,x_{k+1},\ldots,x_n) \in \SL_n(\Ik)$ allows us to assume further that $k=1$.  Now write $M=x_2^{d_2} \cdots x_n ^{d_n}$ for some $d_2,\ldots,d_n \in \IN$. 

\ \\
\noindent{\bf Case 1: } $\Ik$ is infinite, or $\Ik=\IF_q$ and $(q-1) \nmid (d_i+1)$ for some $2\leq i \leq n$.  

In this case there exists $b \in \Ik^*$  such that $b^{d_i+1} \neq 1$.   Set $c=\frac{a}{1-b^{d_i+1}}$ and compute
$$\epsilon _{1,aM} = \delta _{1,i,b} \epsilon _{1,cM}^{-1} \delta _{i,1,b} \epsilon _{1,cM} \in \SLIN_n(\Ik).$$

\noindent{\bf Case 2: } $\Ik=\IF_q$ for some $q=p^s$, and $(q-1) \mid (d_i+1)$ for each $2 \leq i \leq n$. 

We induct on $\deg M = d_2+\cdots+d_n$.  Note that since $a \in \Ik^*$, $\epsilon _{1,aM} \neq \id$, so Corollary \ref{cor:singleAffine} establishes the base case of $\deg M \leq 1$.  

\noindent{\bf Case 2 (a): } $p \nmid (d_{j}+1)$ for some $2 \leq j \leq n$.  

Setting $g=\frac{a}{d_j+1}\left(x_jM\right)\epsilon _{j,1} -\frac{a}{d_j+1}\left(x_jM\right)-aM$, a straightforward computation shows that $\deg g < \deg M$ and
$$\epsilon _{1,aM} = \epsilon _{1,-g} \epsilon _{j,-1} \left(\epsilon _{1,\frac{a}{d_j+1} x_j M}  \epsilon _{j,1} \epsilon _{1,\frac{a}{d_j+1}x_j M} ^{-1}\right). $$
By the inductive hypothesis, $\epsilon _{1,-g} \in \SLIN_n(\Ik)$, so $\epsilon _{1,aM} \in \SLIN_n(\Ik)$ as well.

\noindent{\bf Case 2 (b): } $p \mid (d_j+1)$ for each $2 \leq j \leq n$.

Let $2 \leq k \leq n$ be such that $d_k>1$.  Note that $p \nmid d_k$, and thus ${p+d_k \choose p} \neq 0$; so since $\Ik$ is finite we can choose $b \in \Ik$ such that $b^p{p+d_k \choose p} = a$.  
Then setting $f = \left(x_k^pM\right)\epsilon _{k,b}-x_k^pM$, we have
$$\epsilon _{1,f} = \epsilon _{k,-b} \left(\epsilon _{1,x_k^pM} \epsilon _{k,b} \epsilon _{1,x_k^pM}^{-1}\right) \in \SLIN_n(\Ik).$$

\begin{claim}
Let $S$ be the set of tuples $(r_2,\ldots,r_n) \in \IN^{n-1}$ satisfying either
\begin{enumerate}
\item $r_{2}+\cdots+r_{n} < \deg M$, or
\item $(q-1) \nmid (r_k+1)$.
\end{enumerate}
Then $$f=aM + \sum _{(r_2,\ldots,r_n) \in S} c_{r_2,\ldots,r_n} x_2^{r_2}\cdots x_n^{r_n}$$ for some $c_{r_2,\ldots,r_n} \in \Ik$.
\end{claim}
\begin{proof}
It is straightforward to compute that
$$f=\left(\frac{M}{x_k^{d_k}}\right) \sum _{i=1} ^{d_k+p} {d_k+p \choose i} b^i x_k ^{d_k+p-i}.$$
Note that $\frac{M}{x_k^{d_k}} \in \Ik[x_2,\ldots,\hat{x}_k,\ldots,x_n]$ and that $\deg _{x_j} \left( \frac{M}{x_k^{d_k}}\right)=d_j$ for $j \neq 1,k$.  
By assumption $d_k+1 \equiv 0 \pmod {(q-1)}$; since $q>p$ (by hypothesis), we thus have $r+1 \not \equiv 0 \pmod{(q-1)}$ for any $d_k<r\leq d_k+p$.
\end{proof}
By the induction hypothesis and Case 1 above, we have $\epsilon _{1,f-aM} \in \SLIN_n(\Ik)$.  Thus $\epsilon _{1,aM} = \epsilon _{1,f} \epsilon _{1,f-aM}^{-1} \in \SLIN_n(\Ik)$ as required.

\end{proof}

The analogous statement for $\GLIN_n(\Ik)$ is due to Maubach and Poloni.
\begin{theorem}[\cite{Maubach-Poloni}, Corollary 4.4] \label{thm:GLIN}
Let $\Ik$ be any field other than $\IF_2$. Then 
$$\GLIN_n(\Ik)=\langle \GL_n(\Ik)\rangle^G = \langle \TA_n(\Ik)\rangle ^G.$$
\end{theorem}

A few words are in order about the differences in these two statements.  The proof of Theorem \ref{thm:GLIN} is quite simple, namely the observation that, for any monomial $M \in \Ik[x_2,\ldots,x_n]$ and $a \in \Ik^*$, $\epsilon _{1,aM}=\delta _{1,2} ^{-1} (\epsilon _{1,2aM}^{-1} \delta _{1,2} \epsilon _{1,2aM}) \in \GLIN_n(\Ik)$.  However, $\delta _{1,2} \notin \SL_n(\Ik)$, so this approach had to be adapted (see Case 1 above), resulting in additional technical cases.

We note that Theorem \ref{thm:GLIN} shows $\GLIN_n(\IF_p) = \langle \TA_n(\IF_p) \rangle ^G$; however, it remains (to our knowledge) an open question whether $\SLIN_n(\IF_p) = \langle \EA_n(\IF_p)\rangle ^S$.  In particular, the simplest example where our proof of Theorem \ref{thm:SLIN} breaks down for $\IF_p$ prompts us to ask
\begin{question} Is $(x_1+x_2^5, x_2) \in \SLIN_2(\IF_3)$?
\end{question}

We conclude this section with two frequently used, but simple, observations.
\begin{lemma}\label{reductionlemma}
Let $\phi, \theta \in \SA_n(\Ik)$.  If $\phi$ is normally co-tame and $\phi \in \langle \theta \rangle ^S$, then $\theta$ is normally \cotame.
\end{lemma}
\begin{proof}
$\langle \theta \rangle ^S \supset  \langle \phi \rangle ^S \supset \langle \EA_n(\Ik) \rangle ^S$.
\end{proof}

Combining this with Theorem \ref{thm:SLIN}, we can thus characterize $\SLIN_n(\Ik)$ as the normal subgroup generated by a single automorphism that is both tame and normally co-tame.
\begin{corollary}\label{cotame=} Let $\Ik$ be any field other than $\IF_p$ for a prime $p$.  If $\phi \in \SA_n(\Ik)$ is both tame and normally co-tame, then $\langle \phi \rangle ^S = \SLIN_n(\Ik)$.
\end{corollary}

\section{Main Results}
Throughout this section, we assume that $\Ik$ is a field of characteristic zero  The goal of this section is to prove

\begin{theorem}\label{thm:main}
Let $\Ik$ be a field of characteristic zero.  If $\theta \in \SA_n(\Ik)$ is $m$-triangular for some $m \leq 4$ and $\theta \neq \id$, then $\theta$ is normally co-tame.
\end{theorem}
\begin{proof}
This follows from Theorems \ref{thm:triangular}, \ref{thm:bitriangular}, \ref{thm:3triangular}, and \ref{thm:4triangular} below.
\end{proof}

We begin with two lemmas for handling degenerate cases (cf. the  translation degenerate maps introduced in \cite{Edo-Lewis17}).  The first one is easy to check, but is also a consequence of Lemma \ref{lem:conjugateparabolic}.  We emphasize to the reader that this is the only place we rely on the assumption of $\Ik$ having characteristic zero.
\begin{lemma}\label{lem:noid}
Let $\phi \in \GA_n(\Ik)$.  Then $\epsilon ^{-1} \phi ^{-1} \epsilon \phi = \id$ for every $\epsilon \in \Tr_n(\Ik)$ if and only if $\phi \in \Tr_n(\Ik)$. 
\end{lemma}

\begin{lemma}\label{lem:conjugateparabolic}
Let $\phi \in \GA_n(\Ik)$, and let $\alpha \in \GL_n(\Ik)$.  Fix $1 \leq k \leq n$, and for any $c \in \Ik$, set $\gamma _c = \alpha \epsilon _{k,c} \alpha ^{-1}$.  If $\gamma _c ^{-1} \phi ^{-1} \gamma _c \phi = \id$ for every $c \in \Ik$, then $\lambda \phi \lambda^{-1}\in \PA_n(\Ik)$  for some $\lambda \in \SL_n(\Ik)$.
\end{lemma}

\begin{proof}
First, we note that we may assume $\alpha = \id$, as $$\gamma _c ^{-1} \phi ^{-1} \gamma _c \phi = \alpha \left( \epsilon _{k,c}^{-1} (\alpha ^{-1} \phi ^{-1} \alpha) \epsilon _{k,c} (\alpha ^{-1} \phi \alpha )\right) \alpha ^{-1}.$$
So we now have $\gamma _c = \epsilon _{k,c}$, and note that after writing $\phi = (H_1,\ldots,H_n)$, we can rewrite the assumption $\epsilon _{k,c} ^{-1} \phi ^{-1} \epsilon_{k,c} \phi = \id$ as
\begin{equation}
H_i(x_1,\ldots,x_{k-1},x_k+c,x_{k+1},\ldots,x_n) = H_i(x_1,\ldots,x_n) + \delta _{i,k} c \label{eq1}
\end{equation}
where $\delta _{i,k}$ is the Kroenecker delta.    Write $H_i = \sum _{j=0} ^d P_{i,j} x_k ^j$ for some $P_{i,j} \in \Ik[\hat{x}_k]$, and let $F_i(z)= H_i(x_1,\ldots,x_{k-1},x_k+z,x_{k+1},\ldots,x_n) - H_i(x_1,\ldots,x_n) \in \Ik^{[n]}[z]$.  Then we compute
\begin{align*}
F_i(z) &= \sum _{j=0} ^d P_{i,j} (x_k+z)^j -\sum _{j=0} ^d P_{i,j} x_k ^j\\
&= \sum _{j=0} ^d P_{i,j} \sum _{m=1} ^j {j \choose m} x_k ^{j-m} z^m \\
&= \sum _{m=1} ^d z^m \sum _{r=0} ^{d-m} {r+m \choose m} P_{i,m+r} x_k ^r.
\end{align*}

Note that \eqref{eq1} implies $F_i(c) = \delta _{i,k}c$ for all $c \in \Ik$; since $\Ik$ is infinite, we must have $F_i(z)=\delta _{i,k}z$ as polynomials, and since $\Ik$ has characteristic zero, we must have $d=1$ and $P_{i,1}=\delta _{i,k}$, which implies that $H_i - \delta _{i,k} x_k  \in \Ik[\hat{x}_k]$.  Letting $$\pi = (-x_1,x_2,\ldots,x_n)(x_1,\ldots,x_{k-1},x_n,x_{k+1},\ldots,x_{n-1},x_k) \in \SL_n(\Ik),$$ we then have $\pi \phi \pi ^{-1} \in \PA_n(\Ik)$.
\end{proof}

Before continuing on to triangular automorphisms, we remark that by Corollary \ref{cotame=}, in the proofs of Theorems \ref{thm:triangular}, \ref{thm:bitriangular}, \ref{thm:3triangular}, and \ref{thm:4triangular},  since the classes of interest are tame automorphisms, it suffices to show that the respective maps are normally co-tame.
\begin{theorem}\label{thm:triangular}
Let $\tau \in \SA_n(\Ik) \cap \BA_n(\Ik)$.  If $\tau \neq \id$, then  $\langle \tau \rangle ^S = \SLIN_n(\Ik)$, and in particular $\tau$ is normally co-tame.
\end{theorem}
\begin{proof}

We induct on $\vd(\tau)$.  If $\vd \tau =(0,\ldots,0)$, then $\tau \in \Df_n(\Ik)$ and by Corollary \ref{cor:singleAffine} we have $\langle \tau \rangle ^S = \SL_n(\Ik)^S$.  
Otherwise, by Lemma \ref{lem:noid}, choose $\gamma \in \Tr_n(\Ik)$ such that $\tau _0 := \gamma ^{-1} \tau ^{-1} \gamma \tau \neq \id$.  Note that $\tau _0 \in \BA_n(\Ik) \cap \SA_n(\Ik)$, and by Lemma \ref{lem:vdreduce} $\vd (\tau _0) \llex \vd(\tau)$.  The induction hypothesis gives that $\tau _0$ is normally co-tame, and thus $\tau$ is also normally co-tame by Lemma \ref{reductionlemma}.
\end{proof}

\begin{corollary}\label{cor:parabolic}
Let $\phi \in \SA_n(\Ik) \cap \PA_n(\Ik)$ be parabolic and $\alpha \in \GL_n(\Ik)$. If $\phi \neq \id$, then $\alpha \phi\alpha ^{-1}$ is normally co-tame.
\end{corollary}
\begin{proof}
First, we note that we may assume $\alpha = \id$.  Indeed, write $\alpha = \alpha _0 \lambda$ for some $\alpha _0 \in \SL_n(\Ik)$ and $\lambda \in \D_n(\Ik)$.  Then $\lambda \phi \lambda ^{-1} = \alpha _0 ^{-1} \left(\alpha \phi \alpha ^{-1} \right) \alpha _0 \in \langle \alpha \phi \alpha ^{-1} \rangle ^S$, and since $\lambda \in \D_n(\Ik)$, we have $\lambda \phi \lambda ^{-1} \in \PA_n(\Ik)$.

So it suffices to show that $\phi \in \PA_n(\Ik)$ is normally co-tame.  Write $\phi = \tau \theta$ for some $\tau = (x_1,\ldots,x_{n-1},x_n+P_n(x_1,\ldots,x_{n-1}) ) \in \BA_n(\Ik)$ and $\theta = (H_1(x_1,\ldots,x_{n-1}),\ldots,H_{n-1}(x_1,\ldots,x_{n-1}),x_n) \in \SA_{n-1}(\Ik)$.  Note that if $\theta = \id$, then $\phi=\tau \in \BA_n(\Ik)$ is normally co-tame by Theorem \ref{thm:triangular}.  Otherwise, choose $1 \leq i \leq n$ with $H_i \neq x_i$, and compute
$$\epsilon _{n,x_i} ^{-1} \phi^{-1} \epsilon _{n,x_i} \phi  = \epsilon _{n,x_i}^{-1} \theta ^{-1} \epsilon _{n,x_i} \theta  = \epsilon _{n, H_i-x_i}.$$
Since $H_i-x_i \neq 0$, $\epsilon _{n,H_i-x_i}\neq \id$.  Moreover $\epsilon _{n,H_i-x_i} \in \BA_n(\Ik)$ and is thus normally co-tame by Theorem \ref{thm:triangular}.
\end{proof}
This combined with Lemma \ref{lem:conjugateparabolic} yields the following useful result.
\begin{corollary}\label{cor:choosec}
Let $\phi \in \SA_n(\Ik)$ and $\alpha \in \GL_n(\Ik)$.  Either $\phi$ is normally co-tame, or there exists $c \in \Ik^*$ such that, setting $\gamma = \alpha \epsilon _{n,c} \alpha ^{-1}$, $\gamma ^{-1} \phi ^{-1} \gamma \phi \neq \id$.
\end{corollary}

\begin{theorem}\label{thm:bitriangular}
Let $\phi \in \SA_n(\Ik)$ be two-triangular.  If $\phi \neq \id$, then $\langle \phi \rangle ^S = \SLIN_n(\Ik)$, and in particular $\phi$ is normally co-tame.
\end{theorem}
\begin{proof}
Since $\phi$ is two-triangular, we may write $\phi = \alpha _0 \tau _1 \alpha _1 \tau _2 \alpha _2$ for some $\alpha _i \in \Af_n(\Ik)$ and $\tau _i \in \BA_n(\Ik)$; but noting that $\Af_n(\Ik)=\GL_n(\Ik) \ltimes \Tr_n(\Ik)$, and $\Tr_n(\Ik) \leq \BA_n(\Ik)$, we may assume further that each $\alpha _i \in \GL_n(\Ik)$.  Moreover, by a standard argument we may assume additionally that each $\alpha _i \in \SL_n(\Ik)$ and $\tau _i \in \SA_n(\Ik) \cap \BA_n(\Ik)$.

We also note that by Lemma \ref{reductionlemma} we may assume $\alpha _2 = \id$, as $$(\alpha _2  \alpha _0) \tau _1 \alpha _1 \tau _2 = \alpha _2  \phi \alpha _2 ^{-1} \in \langle \phi \rangle ^S.$$

Now by Corollary \ref{cor:choosec}, we may assume there exists $c \in \Ik^*$ such that,  letting $\gamma = \alpha _1 \epsilon _{n,c} \alpha _1 ^{-1} \in \Tr_n (\Ik)$, $\gamma ^{-1} \phi ^{-1} \gamma \phi \neq \id$.  So we set $\tilde{\phi} =  \left(\gamma ^{-1} \phi ^{-1} \gamma\right) \phi \in \langle \phi \rangle ^S$ and compute
$$ \tilde{\phi} = \gamma ^{-1} \phi ^{-1} \gamma \phi = \gamma ^{-1} \tau _2 ^{-1} \alpha _1 ^{-1} \tau _1 ^{-1} \epsilon _{n,1} \tau _1 \alpha _1 \tau _2  =\gamma  ^{-1} \tau _2 ^{-1} \alpha _1 ^{-1} \epsilon _{n,1} \alpha _1 \tau _2$$

Since $\alpha _1 ^{-1} \epsilon _{n,1} \alpha _1 \in \Tr_n(\Ik) \leq \BA_n(\Ik)$, we thus have $\tilde{\phi} \in \BA_n(\Ik) \cap \SA_n (\Ik)$, and thus $\langle \phi \rangle ^S \geq \langle \tilde{\phi} \rangle ^S \geq \SLIN_n(\Ik)$ by Theorem \ref{thm:triangular}. 
\end{proof}

\begin{theorem}\label{thm:3triangular}
Let $\phi \in \SA_n(\Ik)$ be three-triangular.  If $\phi \neq \id$, then $\langle \phi \rangle ^S = \SLIN_n(\Ik)$, and in particular $\phi$ is normally co-tame.
\end{theorem}
\begin{proof}
As in the proof of Theorem \ref{thm:triangular}, we may assume $\phi = \alpha _0 \tau _1 \alpha _1 \tau _2 \alpha _2 \tau _3$ for some $\alpha _i \in \SL_n(\Ik)$ and $\tau _i \in \SA_n(\Ik) \cap \BA_n(\Ik)$.  We induct on  $\vd(\tau _2)$; if $\vd (\tau _2) =(0,\ldots, 0)$, then $\tau _2 \in \Df_n(\Ik)$ in which case $\phi$ is two triangular and thus normally co-tame by Theorem \ref{thm:bitriangular}.  

So we now assume $(0,\ldots,0) \llex \vd(\tau _2)$. By Corollary \ref{cor:choosec}, we may assume there exists $c \in \Ik^*$ such setting $\gamma = \alpha _0 \epsilon _{n,1} \alpha _0 ^{-1} \in \Tr_n (\Ik)$ and $\tilde{\phi} = \left(\gamma^{-1} \phi ^{-1} \gamma\right)  \phi \in \langle \phi \rangle ^S$, we have $\tilde{\phi} \neq \id$.  We then compute
\begin{align*}
\tilde{\phi} &= \gamma^{-1} \phi ^{-1} \gamma \phi  \\
&= \gamma ^{-1}  \tau _3 ^{-1} \alpha _2 ^{-1} \tau _2 ^{-1} \alpha _1 ^{-1} \tau _1 ^{-1} \epsilon _{n,c} \tau _1 \alpha _1 \tau _2 \alpha _2 \tau _3 \\
&= \gamma ^{-1}  \tau _3 ^{-1} \alpha _2 ^{-1} \tau _2 ^{-1} \alpha _1 ^{-1}  \epsilon _{n,c}  \alpha _1 \tau _2 \alpha _2 \tau _3 
\end{align*}

Note that $\alpha _1 ^{-1} \epsilon _{n,c} \alpha _1 \in \Tr_n(\Ik)$, so setting $\tilde{\tau}_2 = \tau _2 ^{-1} \alpha _1 ^{-1} \epsilon _{n,c} \alpha _1 \tau _2$, we see $\tilde{\tau} _2 \in \BA_n(\Ik)$ with $\vd(\tilde{\tau} _2 ) < \vd (\tau _2)$ by Lemma \ref{lem:vdreduce}.  Thus we have
$$\tilde{\phi} = \gamma ^{-1} \tau _3 ^{-1} \alpha _2 ^{-1} \tilde{\tau}_2 \alpha _2 \tau _3$$
and we see $\tilde{\phi}$ is 3-triangular, and thus $\langle \tilde{\phi} \rangle ^S \geq \SLIN_n(\Ik)$ by the induction hypothesis.
\end{proof}

\begin{theorem}\label{thm:4triangular}
Let $\phi \in \SA_n(\Ik)$ be four-triangular.  If $\phi \neq \id$, then $\langle \phi \rangle ^S = \SLIN_n(\Ik)$, and in particular $\phi$ is normally co-tame.
\end{theorem}
\begin{proof}
As in the proofs of Theorems \ref{thm:bitriangular} and \ref{thm:3triangular}, we may assume $\phi =  \tau _1 \alpha _1 \tau _2 \alpha _2 \tau _3\alpha _3 \tau _4 \alpha _4$ for some $\alpha _i \in \SL_n(\Ik)$ and $\tau _i \in \SA_n(\Ik) \cap \BA_n(\Ik)$.  

\begin{claim}We may assume further that $\alpha _3 = \alpha _2 ^{-1}$, $\tau _4 = \tau _2 ^{-1}$, and $\alpha _4 = \alpha _1 ^{-1}$.
\end{claim}
\begin{proof}[Proof of claim]
To establish this claim, we first note that by Corollary \ref{cor:choosec}, we may choose $c \in \Ik$ such that, letting  $\gamma = \alpha _4^{-1} \epsilon _{n,c} \alpha _4  \in \Tr_n (\Ik)$ and $\tilde{\phi} = \gamma \phi  \gamma^{-1}  \phi ^{-1}$, we have $\tilde{\phi} \neq \id$.  We then compute
\begin{align*}
\tilde{\phi} &= \gamma \phi   \gamma^{-1} \phi ^{-1}  \\
%&= \gamma ^{-1} \tau _4 ^{-1} \alpha _3 ^{-1} \tau _3 ^{-1} \alpha _2 ^{-1} \tau _2 ^{-1} \alpha _1 ^{-1} \tau _1 ^{-1} \epsilon _{n,c} \tau _1 \alpha _1 \tau _2 \alpha _2 \tau _3 \alpha _3 \tau _4 \\
%&= \gamma ^{-1} \tau _4 ^{-1} \alpha _3 ^{-1} \tau _3 ^{-1} \alpha _2 ^{-1} \tau _2 ^{-1} \alpha _1 ^{-1} \epsilon _{n,c} \alpha _1 \tau _2 \alpha _2 \tau _3 \alpha _3 \tau _4 
&= \gamma \tau _1 \alpha _1 \tau _2 \alpha _2 \tau _3\alpha _3 \tau _4 \alpha _4 \gamma ^{-1} \alpha _4 ^{-1} \tau _4 ^{-1} \alpha _3 ^{-1} \tau _3 ^{-1} \alpha _2 ^{-1} \tau _2 ^{-1} \alpha _1 ^{-1} \tau _1 ^{-1} \\
&= \gamma \tau _1 \alpha _1 \tau _2 \alpha _2 \tau _3\alpha _3 \tau _4 \epsilon _{n,c}^{-1} \tau _4 ^{-1} \alpha _3 ^{-1} \tau _3 ^{-1} \alpha _2 ^{-1} \tau _2 ^{-1} \alpha _1 ^{-1} \tau _1 ^{-1} \\
&= \gamma \tau _1 \alpha _1 \tau _2 \alpha _2 \tau _3\alpha _3  \epsilon _{n,c}^{-1} \alpha _3 ^{-1} \tau _3 ^{-1} \alpha _2 ^{-1} \tau _2 ^{-1} \alpha _1 ^{-1} \tau _1 ^{-1}. 
\end{align*}

Note that $\alpha _3 \epsilon _{n,c}^{-1} \alpha _3^{-1}  \in \Tr_n(\Ik)$, so setting $\tilde{\tau}_3 = \tau _3 \alpha _3 \epsilon _{n,c}^{-1} \alpha _3 ^{-1} \tau _3 ^{-1} $, we see $\tilde{\tau} _3 \in \BA_n(\Ik)$.  Thus we have
$$\tilde{\phi} = \gamma \tau _1 \alpha _1 \tau _2 \alpha _2 \tilde{\tau}_3 \alpha _2 ^{-1} \tau _2 ^{-1} \alpha _1 ^{-1} \tau _1 ^{-1}.$$
Finally, observe that $\tau _1^{-1} \tilde{\phi} \tau _1  \in \langle \tilde{\phi} \rangle ^S \leq \langle \phi \rangle ^S$, and that
$$\tau _1 ^{-1} \tilde{\phi} \tau _1  =  (\tau _1 ^{-1} \gamma  \tau _1 ) \alpha _1 \tau _2 \alpha _2 \tilde{\tau} _3 \alpha _2 ^{-1} \tau _2 ^{-1} \alpha _1 ^{-1}.$$
Observing that $\tau _1 ^{-1} \gamma  \tau _1  \in \BA_n(\Ik)$, we have $\tau _1 ^{-1} \tilde{\phi} \tau _1 $ is in the claimed form, and by Lemma \ref{reductionlemma} it suffices to show this map is normally co-tame.
\end{proof}

We now assume  $\phi = \tau _1 \alpha _1 \tau _2 \alpha _2 \tau _3 \alpha _2 ^{-1} \tau _2 ^{-1} \alpha _1 ^{-1}$.  We will induct on $\vd (\tau _3)$, with the $\vd(\tau_3)=(0,\ldots,0)$ case reducing to Theorem \ref{thm:3triangular}.
Once again appealing to Corollary \ref{cor:choosec}, choose $c \in \Ik$ such that setting
\begin{align*}
\gamma &= \alpha _1 \epsilon _{n,c} \alpha _1 ^{-1} &\text{and}&  &\tilde{\phi} = \left(\gamma^{-1} \phi  \gamma \right) \phi ^{-1} \in \langle \phi \rangle ^S.
\end{align*}
we have $\tilde{\phi} \neq \id$.
Then we compute
\begin{align*}
\tilde{\phi} &= \gamma ^{-1} \tau _1 \alpha _1 \tau _2 \alpha _2 \tau _3 \alpha _2 ^{-1} \epsilon _{n,1} \alpha _2 \tau _3 ^{-1} \alpha _2 ^{-1} \tau _2 ^{-1} \alpha _1 ^{-1} \tau _1 ^{-1}
\end{align*}

Let $\tilde{\tau _3} = \tau _3 \alpha _2 ^{-1} \epsilon _{n,1} \alpha _2 \tau _3 ^{-1}$ and let $\tilde{\tau}_1 = \tau _1 ^{-1} \gamma ^{-1} \tau _1$; note that $\tilde{\tau}_1, \tilde{\tau} _3 \in \BA_n(\Ik)$ and $\vd(\tilde{\tau}_3) <\vd(\tau_3)$ by Lemma \ref{lem:vreduce}.  Then we have 
$$\tau _1 ^{-1} \tilde{\phi} \tau _1 = \tilde{\tau} _1 \alpha _1 \tau _2 \alpha _2 \tilde{\tau _3} \alpha _2 ^{-1} \tau _2 ^{-1} \alpha _1 ^{-1} $$
and the induction hypothesis completes the proof.
\end{proof}

\begin{theorem}\label{thm:exponential}
Let $\Ik$ be a field of characteristic zero.
Let $D$ be a nonzero triangular derivation and $F \in \ker D$.  Then $\exp(FD)$ is normally cotame.
\end{theorem}
\begin{proof}
Note that since $D$ is triangular, we have $$\epsilon _{n,1}^{-1}\exp(-FD)\epsilon _{n,1}\exp(FD)=\exp\left( (F-(F)\epsilon _{n,1})D \right).$$

If  $\deg _{x_n} F > 0$, then $\deg _{x_n}(F-(F)\epsilon _{n,1})=\deg _{x_n}F - 1$, so inducting downwards on $\deg _{x_n} F$, we are left to deal with the case that $F \in \Ik[x_1,\ldots,x_{n-1}]$.  But in this case, either $\exp(FD)$ is triangular, or there exists $1 \leq i \leq n-1$ with $(x_i)\exp(FD) = x_i+Q$ for some nonzero $Q \in \Ik[x_1,\ldots,x_{n-1}]$.  Then, letting $$\epsilon _{n,x_i}^{-1}\exp(-FD)\epsilon _{n,x_i} \exp(FD)=\epsilon _{n,Q}$$
we see that $\exp(FD)$ is normally co-tame since $\epsilon _{n,Q}$ is elementary.
\end{proof}

\begin{theorem}\label{thm:triangularexponential}
Let $\Ik$ be a field of characteristic zero.
Let $D$ be a nonzero triangular derivation and $F \in \ker D$; let $\tau \in \BA_n(\Ik)$ and $\alpha \in \GL_n(\Ik)$.  Then $\tau \alpha \exp(FD)$ is normally co-tame.
\end{theorem}
\begin{proof}
Let $\phi = \tau \alpha \exp(FD)$.  By Lemma \ref{reductionlemma} it suffices to show that $\phi _0 = \epsilon _{n,1} ^{-1} \phi^{-1} \epsilon _{n,1} \phi$ is normally co-tame.  Applying Lemma \ref{reductionlemma} once more, it suffices to show that $\phi _1 = \exp(FD) \phi _0 \exp(-FD)$ is normally co-tame.  So we compute $\phi _1$, letting $\gamma = \alpha ^{-1} \epsilon _1 \alpha \in \Tr_n(\Ik)$:
\begin{align*}
\phi _ 1 &=   \exp(FD) \left( \epsilon _{n,1} ^{-1} \exp(-FD) \alpha ^{-1} \tau ^{-1} \epsilon _{n,1} \tau \alpha \exp(FD) \right) \exp(-FD) \\
&=\exp(FD) \epsilon _{n,1} ^{-1} \exp(-FD) \gamma.
\end{align*}
Now, letting $G=(F)\epsilon _{n,1}^{-1}-F$ as in the proof of Theorem \ref{thm:exponential}, we have
$$\phi _1 = \epsilon _{n,1} ^{-1} \exp(GD) \gamma.$$
But then, applying Lemma \ref{reductionlemma} once more, it suffices to show $\phi _2 = \epsilon _{n,1} \phi _1 \epsilon _{n,1}^{-1} \phi _1 ^{-1}$ is normally co-tame, so we compute 
$$\phi _2  = \exp(GD) \gamma \epsilon _{n,1}^{-1}  \gamma ^{-1} \exp(-GD) \epsilon _{n,1} 
= \exp(GD) \epsilon _{n,1}^{-1} \exp(-GD) \epsilon _{n,1}.$$
But letting $H= G-(G)\epsilon _{n,1}$, we have $\phi _2 = \exp(HD)$ which is normally co-tame by Theorem \ref{thm:exponential}.
\end{proof}

\begin{example}\label{ex:vdE}
The automorphism $$ (x_1,x_2+x_1^3,x_3-x_2(x_1x_3+x_2x_4),x_4+x_1(x_1x_3+x_2x_4)) \in \SA_4(\IC)$$ is normally co-tame by Theorem \ref{thm:triangularexponential}, as it can be written as $(x_1,x_2+x_1^3,x_3,x_4) \exp(FD)$ where $F=x_1x_3+x_2x_4$ and $D= -x_2\frac{\partial}{\partial x_3}+x_1 \frac{\partial}{\partial x_4}$.  This automorphism is conjugate (by a permutation) to van den Essen's counterexample \cite{vandenEssen} to Problem \ref{prob:Meister}. 
\end{example}

\subsubsection*{Acknowledgements}
The author would like to thank the referees for a number of helpful comments, and for pointing out the recent paper \cite{LamyPrzytycki}.

\bibliography{bibliography}
\bibliographystyle{siam}

\end{document}